\documentclass[11pt]{article}
\usepackage[ansinew]{inputenc}
\usepackage{graphicx}
\usepackage{color}
\usepackage{times}
\usepackage{pdfsync}
\usepackage{enumerate,latexsym}
\usepackage{latexsym}
\usepackage{amsmath,amssymb}
\usepackage{graphicx}

\usepackage{amsthm}

\def\rth{\mathbb{R}^3}
\def\R{\mathbb{R}}

\def\oB{\ov{\mathbb{B}}}
\def\B{\mathbb{B}}

\def\S{\Sigma}

\def\cC{\mathcal{C}}

\newcommand{\nc}{\newcommand}
\newcommand{\ben}{\begin{enumerate}}
\newcommand{\bit}{\begin{itemize}}
\newcommand{\een}{\end{enumerate}}
\newcommand{\eit}{\end{itemize}}
\newcommand{\wh}{\widehat}
\newcommand{\Int}{\mbox{\rm Int}}
\newcommand{\inj}{\mbox{\rm Inj}}

\newcommand{\cD}{\mathcal{D}}
\newcommand{\ov}{\overline}

\newcommand{\wt}{\widetilde}

\newcommand{\bp}{\begin{proof}}
\newcommand{\ep}{\end{proof}}

\def\a{{\alpha}}

\def\G{{\Gamma}}

\def\de{{\delta}}

\def\ve{{\varepsilon}}

\newtheorem{theorem}{Theorem}[section]
\newtheorem{lemma}[theorem]{Lemma}
\newtheorem{proposition}[theorem]{Proposition}

\newtheorem{remark}[theorem]{Remark}

\newtheorem{corollary}[theorem]{Corollary}
\newtheorem{definition}[theorem]{Definition}

\newtheorem{claim}[theorem]{Claim}

\newcommand{\ed}{\end{document}}
\nc{\bl}{\begin{lemma} }

\nc{\el}{\end{lemma} }

\nc{\bt}{\begin{theorem} }

\nc{\et}{\end{theorem} }
\newcommand{\rc}{ \renewcommand }

\rc{\v}{    \overset{\longrightarrow} }

\nc{\Si}{\Sigma}

\begin{document}

\begin{title}{Chord arc properties for constant mean curvature disks}
\end{title}
\begin{author}
{William H. Meeks III\thanks{This material is based upon
   work for the NSF under Award No. DMS-1309236.
   Any opinions, findings, and conclusions or recommendations
   expressed in this publication are those of the authors and do not
   necessarily reflect the views of the NSF.}
   \and Giuseppe Tinaglia\thanks{The second author was partially
   supported by
EPSRC grant no. EP/M024512/1}}
\end{author}
\date{}
\maketitle
\vspace{-.3cm}

\begin{abstract}
We prove  a chord arc type bound for
disks  embedded in $\rth$ with constant mean curvature that
does not depend on the value of the  mean curvature. This bound is inspired by and
generalizes the weak chord arc bound of Colding and Minicozzi
in Proposition 2.1 of~\cite{cm35} for
embedded minimal disks.
Like in the minimal case, this chord arc bound is a fundamental tool
for studying complete constant mean curvature surfaces embedded  in $\rth$ with
finite topology or with positive injectivity radius.
\end{abstract}

\vspace{.3cm}



\section{Introduction.}
In this paper we apply  results in~\cite{cm35,mr13,mt7,mt13,mt9} to
derive a chord arc bound for compact disks embedded
in $\rth$ with constant mean curvature. For clarity of  exposition,
we will call an oriented surface
$\Sigma$ immersed in $\rth$ an {\it $H$-surface} if it
is {\it embedded}, {\em connected}  and it
has {\it non-negative constant mean curvature $H$}. We will  call an
$H$-surface an {\em $H$-disk} if the $H$-surface is homeomorphic
to a closed unit disk in the Euclidean plane; in general we will
allow an $H$-disk $\Sigma$ to be non-smooth along its boundary.
We remark that this definition of $H$-surface
agrees with the one given in~\cite{mt13,mt9},
but differs from the one given in~\cite{mt7}
where we restrict to the case when $H>0$.

It will be important to distinguish between intrinsic and
extrinsic balls centered at points of  $\Sigma $; given $p\in
\Sigma$ and $R>0$, we will denote by $B_{\Sigma }(p,R)$ (resp.
$\B(p,R)$)  the {\em open} intrinsic (resp. open extrinsic) ball of center
$p$ and radius $R$ and let $\B(R)=\B(\vec{0},R)$, where
$\vec{0}$ is the origin in $\rth$. We will denote by $\ov{B}_{\Sigma }(p,R)$ (resp.
$\oB(p,R)$)  the {\em closed} intrinsic (resp. closed extrinsic) ball of center
$p$ and radius $R$  and let $\oB(R)$ be the closure of $\B(\vec{0},R)$.

\begin{definition} \label{def1.1} {\rm Given a point $p$ on a surface
$\Sigma\subset \rth$, $\Si (p,R)$ denotes
 the closure of the component of
 $\Sigma \cap \B(p,R)$ passing through $p$.}
\end{definition}

We note that if the surface $\Sigma$ in the above definition is
transverse to $\partial \B(p,R)$,
then  $\Si (p,R)$ is  the component of $\Sigma \cap \ov{\B}(p,R)$ passing
through $p$. The main result of this paper is the following theorem.

\begin{theorem}[Weak chord arc property for $H$-disks] \label{thm1.1}
There exists a $\delta_1 \in (0,
\frac{1}{2})$  such that the following holds.

Let $\Si$ be an   $H$-disk in $\rth.$  Then for all
intrinsic closed balls $\ov{B}_\Si(x,R)$ in $\Si-
\partial \Si$: \ben \item $\Si (x,\delta_1 R)$ is a disk with piecewise smooth boundary
$\partial \Sigma(x,\delta_1 R)\subset \partial \B(x,\de_1R)$.
 \item
$
 \Si (x, \delta_1 R) \subset B_\Si (x, \frac{R}{2}).$
\een
\end{theorem}
Theorem~\ref{thm1.1} gives rise to a more standard chord arc type
result that closely resembles the chord arc type result for
0-disks given by Colding and Minicozzi in
Theorem~0.5 in~\cite{cm35}; see
Theorem~1.2 in~\cite{mt13} for this application.

We clarify that Theorem~\ref{thm1.1} in this manuscript depends on the
extrinsic one-sided curvature estimate for $H$-disks, Theorem~1.1 in~\cite{mt9}. 
On the other hand, the intrinsic one-sided curvature estimate for $H$-disks, 
Theorem~4.5 in~\cite{mt9}, relies on Theorem~\ref{thm1.1} in this manuscript.

Other applications of the results in this manuscript can be found in~\cite{mt15}.
\vspace{.2cm}

\vspace{.2cm}

\noindent  {\sc Acknowledgements:} The authors would
like to thank Joaquin Perez  for making the figures that appear in this
paper.

\section{Proof of  Theorem~\ref{thm1.1}.}
The proof of Theorem~\ref{thm1.1} relies on three results that appear
in~\cite{mt13,mt9}, and for the sake of completeness,
we include their statements here. Theorems~\ref{th} and~\ref{thm2.1} are
generalizations of results that were proved by Colding and Minicozzi in the
minimal case in~\cite{cm23}.

\begin{theorem}[One-sided curvature estimate for $H$-disks, Theorem~1.1
in~\cite{mt9}] \label{th} There exist $\ve\in(0,\frac{1}{2})$
and $C >0$ such that for any $R>0$, the following holds.
Let $\cD$ be an $H$-disk such that $$\cD\cap \B(R)\cap\{x_3=0\}
=\O \quad \mbox{and} \quad \partial \cD\cap \B(R)\cap\{x_3>0\}=\O.$$
Then:
\begin{equation} \label{eq1}
\sup _{x\in \cD\cap \B(\ve R)\cap\{x_3>0\}} |A_{\cD}|(x)\leq \frac{C}{R},
\end{equation} where $|A_{\cD}|$ denotes the norm of the second fundamental form
of ${\cD}$. In particular, if
$\cD\cap \B(\ve R)\cap\{x_3>0\}\neq\mbox{\rm \O}$, then $H< \frac{C}{R}$.
\end{theorem}

\begin{theorem}[Limit lamination theorem for $H$-disks,
Theorem~1.1 in~\cite{mt13}] \label{thm2.1}
Fix  $\ve >0$ and let $\{M_n\}_n$ be
a sequence of $H_n$-disks in $\rth$ containing the origin  and such that
$\partial M_n \subset [\rth - \B(n)]$ and
$|A_{M_n} |(\vec{0})\geq \ve$. Then, after replacing  by
some subsequence,  exactly one of the
following two statements hold.
\ben[A.]
\item The surfaces $M_n$ 
converge smoothly with multiplicity one or two on compact
subsets of $\rth$ to a helicoid $M_{\infty}$
containing the origin. Furthermore,
every component $\Delta$ of $M_n\cap \B(1)$  is an open disk
whose closure $\ov{\Delta}$ in $M_n$
is a compact disk with piecewise smooth boundary, and
where the intrinsic distance in $M_n$
between  any two points in  $\ov{\Delta}$ is less than 10.
\item  There are points $p_n\in M_n$ such that
\[
\lim_{n\to \infty}p_n=\vec{0} \text{\, and \,}
\lim_{n\to \infty}|A_{M_n}|(p_n)=\infty,
\] and the
following  hold: \ben
\item The surfaces $M_n$ converge to a foliation
of $\rth$ by planes and the convergence is $C^\alpha$, for any $\a\in(0,1)$,
away from the line containing
the origin and orthogonal to the planes in the foliation. \item
There exists  compact subdomains
$\cC_n$ of $M_n$, $[M_n\cap \oB(1)]\subset \cC_n \subset \B(2)$
and  $\partial \cC_n\subset \B(2)-\oB(1)$, each $\cC_n$
consisting of one or two pairwise disjoint disks, where
each disk component has intrinsic diameter less than 3 and intersects $\B(1/n)$.
Moreover, each connected
component of $M_n\cap \B(1)$ is an open disk whose closure in $M_n$
is a compact disk with piecewise smooth boundary. \een
\een
\end{theorem}

\begin{corollary} [Corollary~4.6 in \cite{mt13}]\label{cest} There exist constants $\ve\in (0,1)$, $C>1$
such that the following holds.  Let
$\Sigma_1$, $\Sigma_2$, $\Sigma_3$ be three pairwise disjoint $H_i$-disks with
$\partial \Sigma_i\subset [ \rth- \B(1)]$ \,for $i=1,2,3$.
If $\,\B(\ve)\cap\Sigma_i\not=\O$ for $i=1,2,3$, then {\large
\[
\sup_{\B(\ve)\cap\Sigma_i,\,i=1,2,3}|A_{\Sigma_i}|\leq C.
\]}\end{corollary}

\subsection{A weak chord arc property for certain $H$-disks.}
 Throughout this  section and the next one, $\Sigma$ will denote a compact
$H$-disk in $\rth$ with piecewise smooth boundary.

The following main result in this section generalizes the similar
Proposition~2.1 in~\cite{cm35} for
minimal disks to certain $H$-disks. The reader should keep in mind
that the convex hull property of minimal surfaces fails in the case of $H$-surfaces
with $H>0$, and this failure contributes to making the proof of the next
proposition and some other results
in this paper more difficult than in the $H=0$ case.
\begin{proposition} \label{cm2.1}  There exists
$\delta_2  \in (0, \frac{1}{2})$ such
that  the following holds.

If $\Sigma$ satisfies $\partial \Si \subset
\partial \B(p,R)$ and  $p \in \Sigma$, then for all $s\in (0,R]$:
\ben
\item $\Sigma(p, \delta_2 s)$ is a disk with piecewise smooth boundary
$\partial \Sigma(p,\delta_2 s)\subset \partial \B(p,\de_2s)$.
\item $ \Sigma (p, \delta_2 s) \subset \ov{B}_\Sigma (p, \frac{s}{2})$.
\een
\end{proposition}
\begin{proof}

Suppose that $\Sigma$ satisfies $\partial \Si \subset
\partial \B(p,R)$ and  $p \in \Sigma$.
We first prove that items~1 and 2 of the
proposition hold  for
some $\delta_2 \in (0, \frac{1}{2})$ in the special case that $s=R$.
Arguing by contradiction, suppose there is no such
universal $\delta_2$. Then there exists
a sequence $\Si(n)$ of $H_n$-disks and a sequence $R_n$
of positive numbers  such that
 \ben[1.]
  \item $\vec{0} \in \Si(n)$.
   \item  $\partial \Si
(n) \subset \partial \B(R_n)$.
\item Either $\Si(n)(\vec{0}, \frac{R_n}{n})$
is not a disk or it is not contained in
$\ov{B}_{\Sigma(n)}(\vec{0}, \frac{R_n}{2})$.  \een
Let $\widetilde{\Si}(n)$ be the sequence of rescaled disks
$\frac{n}{R_n} \Si(n)$, see Figure~\ref{figurepropos7}.
\begin{figure}\begin{center}
\includegraphics[width=5cm]{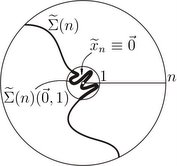}
\caption{The component $\widetilde{\Sigma }(n)(\vec{0},1)$ inside
the rescaled surface $\widetilde{\Sigma }(n)$.}
\label{figurepropos7}\end{center}
\end{figure}
Note that for all $n$, $\partial \widetilde{\Si}(n)\subset \partial \B(n)$
and  $\widetilde{\Sigma }(n)(\vec{0},1)$ is not a disk or it is not
contained in $\ov{B}_{\widetilde{\Sigma}(n)} (\vec{0}, \frac{n}{2})$.

After replacing by a subsequence, one of the following three cases holds:
\begin{enumerate}
\item $\lim_{n \to \infty}\max_{x\in\wt \Si(n)\cap\oB(1)}|A_{\widetilde\Sigma(n)}(x) |=0;$
\item $\lim_{n \to \infty}\max_{x\in\wt \Si(n)\cap\oB(1)}|A_{\widetilde\Sigma(n)}(x) |=\infty ;$
\item $\lim_{n \to \infty}\max_{x\in\wt \Si(n)\cap\oB(1)}|A_{\widetilde\Sigma(n)}(x) |=L\in(0,\infty).$
\end{enumerate}

First consider the case that $\lim_{n \to \infty}\max_{x\in\wt \Si(n)\cap\oB(1)}|A_{\widetilde\Sigma(n)}(x) |=0.$
Then  for $n$ large, $\widetilde{\Sigma
}(n)(\vec{0},1)$ is an almost totally geodesic disk whose diameter is
bounded by $3$ and that is a small graph over its projection to the unit
disk in the $(x_1,x_3)$-plane, which gives a contradiction.

Next, suppose that $\lim_{n \to \infty}\max_{x\in\wt \Si(n)\cap\oB(1)}|A_{\wt \Sigma(n)}(x) |=\infty$
and, after going to a subsequence, let $p_n\in \wt \Sigma(n)\cap\oB(1)$ be a sequence of points such that
\[
\lim_{n\to\infty} p_n=p\in \oB(1)\,\text{ and }\lim_{n \to \infty}|A_{\wt \Sigma(n)}(p_n) |=\infty.
\]
Then, we can apply Theorem~\ref{thm2.1} to the sequence of translated surfaces
$\Si'(n)=\frac13[\wt \Si(n)-p_n]$. In particular, since
$\lim_{n \to \infty}|A_{ \Sigma'(n)}(\vec 0) |=\infty$, Case B of Theorem~\ref{thm2.1} applies.
Note that the origin, as a point contained in $\wt\Sigma(n)$, has become 
the point $-\frac{p_n}3$ in $\Sigma'(n)$ and, by our hypothesis,  $\Sigma'(n)(-\frac{p_n}3,\frac13)$ 
is not a disk or it is not
contained in $\ov{B}_{\Sigma'(n)} (-\frac{p_n}3, \frac{n}{6})$. By Case B of Theorem~\ref{thm2.1},
since $|\frac{p_n}{3}|\leq \frac 13$, we have that
\[
\Sigma'(n)(-\frac{p_n}3,\frac13)\subset [\Sigma'(n)\cap \oB(1)]\subset \cD_n \subset \B(2)
\]
 where $\cD_n$ is a disk with intrinsic diameter bounded by 3 and
 $\partial \cD_n\subset \B(2)-\oB(1)$. Let $\Delta_n$ denote $\Sigma'(n)(-\frac{p_n}3,\frac13)$.
 Since $\Delta_n\subset \cD_n$ and the intrinsic diameter of $\cD_n$ is bounded by 3,
 in order to obtain a contradiction, it suffices to prove that $\Delta_n$ is a disk.
 In the case where $\cD_n$  intersects
$\partial \B(-\frac{p_n}3,\frac13)$ transversely,  $\Delta_n$ is a smooth compact surface
and the following arguments
can be simplified; therefore, on a first reading of the next paragraph
the reader might want to consider this special generic case first.

Since $\Delta_n$ is a two-dimensional semi-analytic
set in $\rth$ and $\Delta_n\cap \partial \B(-\frac{p_n}3,\frac13)$ is an analytic subset of the sphere
$\partial \B(-\frac{p_n}3,\frac 13)$, then, by \cite{lo1}, $\Delta_n$
admits a triangulation  by analytic  simplices, and the interiors of
the 2-dimensional simplices are contained in
$\B(-\frac{p_n}3,\frac 13)$ because otherwise, by analyticity, then
$\Delta_n\subset\partial \B(-\frac{p_n}3,\frac13)$ which is false. Since the inclusion map
of $\cD_n$ is an injective immersion, then it follows that $\Delta_n$
is a semi-analytic subset of $\cD_n$ that can be triangulated with
a finite number of closed 2-dimensional analytic simplices whose
interiors are contained in $\Delta_n\cap\B(-\frac{p_n}3,\frac13)$  and
$\Delta_n\cap \partial \B(-\frac{p_n}3,\frac13)$ is  a connected 1-dimensional
analytic subset of $\cD_n$, where we identify $\cD_n$ with its image in $\rth$;
note that $\Delta_n\cap\partial \B(-\frac{p_n}3,\frac13)$  does not contain any isolated points
by the mean curvature comparison principle.  By the
elementary topology of the disk $\cD_n$ and using arguments as in~\cite{my1}, one can check that
$\Delta_n$ fails to be a
disk with piecewise smooth analytic boundary
if and only if there exists a
simple closed piecewise analytic  curve $\G(n)  $ contained in the 1-dimensional simplicial sub-complex
of $\Delta_n\cap\B(-\frac{p_n}3,\frac13)$ such that  $\G(n) $
does not bound a  disk in $\Delta_n\cap\B(-\frac{p_n}3,\frac13)$. In the case that $\cD_n$
is transverse to $\partial \B(-\frac{p_n}3,\frac13)$, then $\G(n)$ can be chosen to be the boundary curve
of  a   component of $\cD_n\cap (\rth -\B(-\frac{p_n}3,\frac13)) $ that has its 
entire boundary in $\partial \B(-\frac{p_n}3,\frac13)$.

Arguing by contradiction, suppose   $\Delta_n$ is not a compact disk.
Let $D_n$ denote the compact subdisk of $\cD_n$ with
boundary $\G(n)\subset \cD_n\cap \partial \B(-\frac{p_n}3,\frac13)$ and notice
that $D_n\not \subset \oB(-\frac{p_n}3,\frac13)$.  Hence, there is a point $q_n\in D_n$ that
has maximal distance $T_n>\frac13$ from the $-\frac{p_n}3$. Since the boundary of $D_n$ lies
in $\partial B(-\frac{p_n}3,\frac13)$ and $D_n$ lies in $\rth - \oB(-\frac{p_n}3,\frac13)$ near $\partial D_n$,
then $q_n$ is an interior point of $D_n$ not contained in $\B(-\frac{p_n}3,\frac13)$
and $D_n$ lies inside the closed ball $\ov{\B}(-\frac{p_n}3,T_n)$ and intersects
$\partial \B(-\frac{p_n}3,T_n)$ at the point $q_n$.  By the mean curvature
comparison principle applied at the point $q_n$,
the constant mean curvature of $\cD_n$ is at least $1/T_n$.
By Theorem~\ref{thm2.1} the constant
 mean curvature values of the surfaces $\cD_n$ are tending to zero as  $n$ goes
to infinity (portions of the surfaces $\S'(n)$ are converging to planes), then
the interior points $q_n\in D_n\subset \cD_n\subset \B(2)$
are diverging to infinity in $\rth$ as $n$ goes
to infinity, which is a contradiction and proves that $\Delta_n$ must be a disk.

Finally, suppose that $\lim_{n \to \infty}\max_{x\in\wt \Si(n)\cap\oB(1)}|A_{\wt \Sigma(n)}(x) |=L\in(0,\infty)$
and, after going to a subsequence, let $p_n\in \wt \Sigma(n)\cap\oB(1)$ be a sequence of points such that
\[
\lim_{n\to\infty} p_n=p\in \oB(1)\,\text{ and }\lim_{n \to \infty}|A_{\wt \Sigma(n)}(p_n) |=L.
\]
In this case, after going to a subsequence, one of the following sub-cases holds.
\begin{enumerate}
\item[3A.] $\lim_{n \to \infty}\max_{x\in\wt \Si(n)\cap\oB(2)}|A_{\wt \Sigma(n)}(x) |<\infty;$
\item[3B.] $\lim_{n \to \infty}\max_{x\in \wt \Si(n)\cap\oB(2)}|A_{\wt \Sigma(n)}(x) |=\infty$.
\end{enumerate}
If Case 3A holds, then by applying Theorem~\ref{thm2.1} to the sequence of translated surfaces
$\wt\Si(n)-p_n$ we obtain that, after going to a subsequence, $\wt\Si(n)-p_n$ converges  smoothly
with multiplicity one or two on compact
subsets of $\rth$ to a helicoid containing the origin, which is a
surface of negative curvature. This being the case then, for $n$
sufficiently large, $|A_{\wt\Sigma(n)}(\vec 0) |>\ve\in(0,1)$ and Case A of Theorem~\ref{thm2.1}
applies to the sequence of surfaces $\wt\Si(n)$. In particular,  every component
$\Delta$ of $\wt\Si(n)\cap \B(1)$  is an open disk
whose closure $\ov{\Delta}$ in $\wt\Si(n)\cap \B(1)$
is a compact disk with piecewise smooth boundary, and
where the intrinsic distance in $\wt\Si(n)$
between  any two points in  $\ov{\Delta}$ is less than 10. This contradicts our assumption.

If Case 3B  holds then one can obtain a contradiction by arguing similarly to Case 2.
This completes the proof that for some
$\delta_2 \in (0, \frac{1}{2})$ in the case $s=R$.  This fixes the value of $\de_2$.

Fix $s\in (0,R)$. By the arguments in the previous case where $s=R$,
$\Si  (p,s)$ admits an analytic triangulation and it
is the closure in $\S$ of a connected open
surface. Hence by the elementary topology of a disk,
the set $\wh{\Si} [p,s]\subset \S$ that is the closure  of the
complement of the annular component of
$\Si -\Si  (p,s)$ that contains $\partial \Si $ is a
piecewise-smooth  subdisk of $\Si $; also note that
$p\in\wh{\Si} [p,s]$ and $\partial \wh{\Si} [p,s] \subset \partial  \B(s)$.
Applying the previously proved case where $s=R$  to the
$H$-disk $\wh{\Si} [p,s]$ and the subdomain
$\wh{\Si} [p,s](p,s)$ (which is equal to the domain $\Sigma  (p, s)$), one has that
$\Sigma  (p, \delta_2 s)$ is a disk with
$\partial \Sigma(p,\delta_2 s)\subset \partial \B(p,\de_2s)$ and
$\Sigma  (p, \delta_2  s)\subset \ov{B}_{\Si }(p,\frac{s}{2})$.

This finishes the proof of the proposition.
\end{proof}

\subsection {Expanding the scale of being $\delta_2$ weakly chord arc.}
We begin by giving a definition characterizing
certain intrinsic geodesic balls of $\Si$.

\begin{definition}{\rm (Weakly chord arc)  Given
$\delta \in (0, \frac{1}{2}),$ an intrinsic  ball $\ov{B}_\Si(x,R) \subset
\Si$ is said to be $\delta$ {\em  weakly chord arc} if \ben[1.] \item
For all $s \in(0, R)$, $\ov{B}_\Si(x,s) \subset
\Int(\Si)$.
\item For all $s \in(0, R]$, \ben
\item $\Si(x, \delta s)$ is a disk;
\item
$\Si(x, \delta s) \subset \ov{B}_\Si(x, \frac{s}{2}).$ \een \een}
\end{definition}

\begin{remark} \label{remark} {\em Suppose that $x\in \Si$.
Notice that if an intrinsic ball $\ov{B}_\Sigma(x,R) \subset \Sigma - \partial \Si$,
then for any $s \in(0, R)$,  $\ov{B}_\Sigma(x,s)$ is
contained in the interior of $\Si$.
Also note that if $\partial \Si \subset \partial \B(x,R)$,
then Proposition~\ref{cm2.1} implies that
$\ov{B}_{\Si}(x,R)$ is $\de_2$ weakly chord arc.}
\end{remark}

\begin{definition} Given
$\delta \in (0, \frac{1}{2})$ and $x \in \Si - \partial \Si$,
\[
R(x,\delta) = \sup \{ R < {\rm dist} (x, \partial \Si) \mid
{\rm the\; ball} \;\ov{B}_\Si(x,R)\;{\rm  is} \;\delta \;
 {\rm weakly \;chord\; arc}\}.
 \]
\end{definition}

Our definition of the $R(x, \delta)$ function is the
same as the one given in~\cite{mr13} and differs somewhat from the
related $R_\delta(x)$ function defined in~\cite{cm35}.

We now state and prove a proposition that in certain cases
allows us to prove that if a given ball $\ov{B}_\Si(x, R)$ in $\Si$ is $\delta_2$
weakly chord arc, then $\ov{B}_\Si(x, 5R)$ is also $\delta_2$ weakly
chord arc; here, $\delta_2$ is the constant
defined in Proposition~\ref{cm2.1}.
The next result corresponds to the closely
related Proposition~3.4 in~\cite{cm35}
and Proposition~8 in~\cite{mr13}.

\begin{proposition} \label{propcm3.4}  There exists a
constant $C_b > 5$ so that if
$\ov{B}_\Sigma(y, C_bR) \subset \Si - \partial \Si$ satisfies
\[
 ``every \; intrinsic \;subball \; \ov{B}_\Sigma(z, R) \subset \ov{B}_\Si(y,
C_b R) \; is\;  \delta_2 \; weakly \; chord  \; arc,"
\]
then, $\ov{B}_\Si(y, 5 R)$ is $\delta_2$ weakly chord arc. In
particular, $R(y, \delta_2) \geq 5 R$.
\end{proposition}

\begin{proof}  Arguing by contradiction, suppose that
Proposition~\ref{propcm3.4} fails. Then there  exist a sequence of $H_n$-disks $\Sigma(n)$ and constants $
C_n > 5n$, $R_n>0$ satisfying:
\ben[1.]
\item  $\ov{B}_{\Si(n)} (y_n, C_n R_n)
\subset \Si(n) - \partial \Si(n).$
\item Every intrinsic subball
$\ov{B}_{\Sigma(n)} (z, R_n) \subset \ov{B}_{\Sigma(n)} (y_n, C_n R_n)$ is
$\delta_2$ weakly chord arc. \item $\ov{B}_{\Sigma(n)}(y_n, 5R_n)$ is not
$\delta_2$ weakly chord arc.
 \een

Let us first assume that, after passing to a subsequence,
\[
\Sigma (n)(y_n,5R_n)\cap
\partial B_{\Sigma (n)}(y_n,C_nR_n)=\mbox{\O } \quad \mbox{for all $n$.}
\]
 Since
$\ov{B}_{\Sigma (n)}(y_n,C_nR_n)\subset \Sigma (n)-\partial \Sigma (n),$
the above intersection equation implies $\Si(n)(y_n, 5R_n)\subset
\Sigma (n)-\partial \Sigma (n)$. By the arguments in the last paragraph
of the proof of Proposition~\ref{cm2.1}, $\Si(n)$ contains a compact
subdisk $\wh{\Si}(n)\subset \Si-\partial\Si$ with $\partial \wh{\Si}(n)\subset \partial \B(y_n,5R_n)$
and $\Si(n)(y_n, 5R_n)=\wh{\Si}(n)(y_n, 5R_n)$. Since
$\ov{B}_{\Sigma (n)}(y_n, 5R_n)=\ov{B}_{\wh{\Sigma} (n)}(y_n, 5R_n)$,
Remark~\ref{remark} implies
$\ov{B}_{\Sigma (n)}(y_n, 5R_n)$ is $\delta_2$ weakly chord arc,
which is a contradiction to item 3 above.
\begin{figure}
\begin{center}
\includegraphics[width=8.5cm]{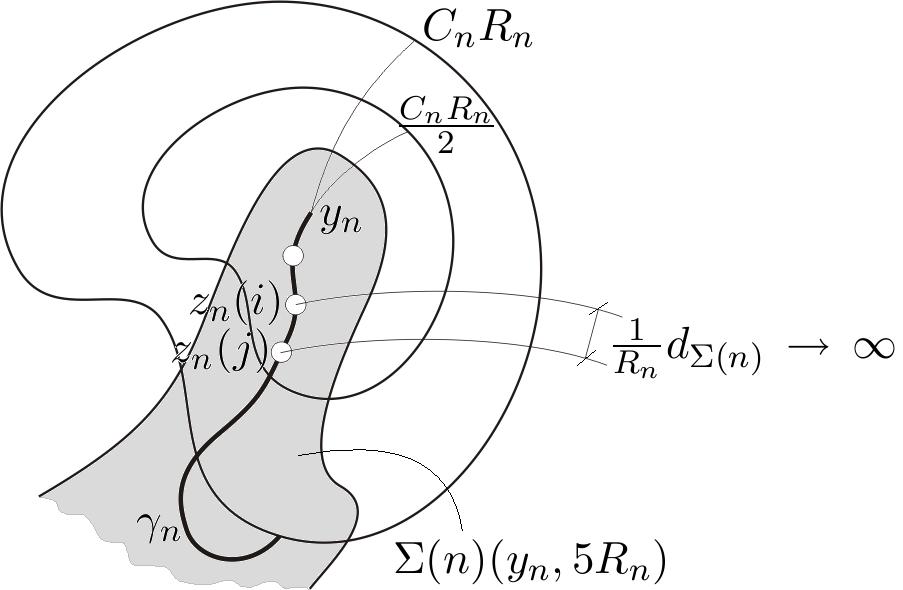}
\caption{$d_{\Sigma (n)}$ denotes the intrinsic distance in $\Sigma
(n)$.} \label{figurepropos9}
\end{center}
\end{figure}
Hence, for the remainder of the proof we shall assume that
\[
\Sigma (n)(y_n,5R_n)\cap \partial
B_{\Sigma (n)}(y_n,C_nR_n)\neq \mbox{\O } \quad \mbox{for all $n$.}
\]

Since
$\Sigma (n)(y_n,5R_n)$ is path connected, we can find a path $\gamma _n
\subset \Sigma (n)(y_n,5R_n)$ starting at $y_n$ and ending at a
point of $\partial B_{\Si(n)}(y_n, C_n R_n)$. Homothetically scale the surfaces $\Sigma(n)$
by $\frac{1}{R_n}$ from the points $y_n$ to obtain new
surfaces $\widetilde{\Si}(n)$ passing through $y_n$;
we will use tilde to denote other related scaled objects as well.
Balls of radius $C_n R_n$ then become
balls of radius $C_n >5n$, and  balls of radius $R_n$ become
balls of radius one.  The corresponding expanded path
$\widetilde{\gamma }_n \subset \wt{\Sigma} (n)(\wt{y}_n,5)\subset \oB(\wt{y}_n,5)$ joins $\wt{y}_n=y_n$
with a point of  $\partial \B(\wt{y}_n,5)$ at an intrinsic
distance $C_n$ from $\wt{y}_n$.
Since $\lim_{n\to\infty}C_n= \infty $, there exists
a subset $\widetilde{S}_n=\{ \widetilde{z}_n(1),\widetilde{z}_n(2),\ldots
,\widetilde{z}_n(k(n))\} \subset \widetilde{\gamma }_n\cap
\ov{B}_{\widetilde{\Sigma }(n)}(\wt{y}_n,\frac{C_n}{2})$ satisfying:
\begin{enumerate}[P1.]
\item $\lim_{n\to\infty} k(n)=\infty$.
\item The intrinsic distance in $\widetilde{\Sigma }(n)$ between any two of the
points of $\widetilde{S}_n$ tends to infinity as $n$ goes to infinity.
\end{enumerate}
In the original scale, we have corresponding finite sets
\[S_n=\{ {z}_n(1),{z}_n(2),\ldots
,{z}_n(k(n))\}\subset
\gamma _n\cap \ov{B}_{\Sigma (n)}(y_n,\frac{C_nR_n}{2});\]
 see
Figure~\ref{figurepropos9}.

We claim that for any $\widetilde{z}\in \widetilde{S}_n$,
$\ov{B}_{\widetilde{\Sigma}(n)}(\widetilde{z},1)$ is $\de_2$ weakly chord arc, which
after scaling, is equivalent to proving that
$\ov{B}_{{\Sigma}(n)}({z},R_n)$ is $\de_2$ weakly chord arc.
Since $\wt{z}\in \ov{B}_{\wt{\Sigma} (n)}(\wt{y}_n,\frac{C_n}{2})$,
$\ov{B}_{\wt{\Sigma} (n)}(\wt{y}_n,{C_n})\subset
(\wt{\Si}(n)-\partial \wt{\Si}(n))$ and $C_n>5$, the intrinsic triangle
inequality implies that $\ov{B}_{\wt{\Sigma} (n)}(\wt{z},1)\subset \ov{B}_{\wt{\Sigma}
(n)}(\wt{y}_n,C_n)$ and so, $\ov{B}_{\Sigma (n)}(z,R_n)\subset \ov{B}_{\Sigma
(n)}(y_n,C_nR_n)$. The main hypothesis in
the statement of the proposition implies
under these conditions that $\ov{B}_{{\Sigma}(n)}({z},R_n)$
is $\de_2$ weakly chord arc. This proves
the desired claim and so, in particular, $\widetilde{\Sigma
}(n)(\widetilde{z},\delta _2)$ is a disk contained in $\ov{B}_{\widetilde{\Sigma
}(n)}(\widetilde{z},\frac12)$.

By the last claim and condition P2 above, the disks $\widetilde{\Sigma
}(\widetilde{z},\delta _2)$ are pairwise disjoint for distinct
points $\widetilde{z}$ of $\widetilde{S}_n$ for $n$ large. At this point
it is useful to make another normalization of the surfaces
$\Si(n)$ by translating them by the vector $-\vec y_n$ and so the points
$y_n$ are now equal to the origin: with this normalization, $\wt{y}_n=y_n=\vec{0}$.
Note that
for any $\widetilde{z}\in \widetilde{S}_n$, the disk $\widetilde{\Sigma
}(n)(\widetilde{z},\delta _2)$ is contained in $\B(6)$. This follows
since $\widetilde{z}\in \oB(5)$ and $\delta_2\in(0,\frac 12)$.
The number of disks
$\widetilde{\Sigma }(n)(\widetilde{z},\delta _2)$ centered at points
$\widetilde{z} \in \widetilde{S}_n$ goes to infinity as $n$ goes to infinity,
by condition P1 in our choice of the points. Hence, after
replacing by a subsequence, there is a point
$q\in \oB(5)$, where the number
of points in $\B(q, \frac{1}{n}) \cap
\widetilde{S}_n$ goes to infinity as $n$ goes to infinity.
In particular, we may assume that for each $n$, there exist
three distinct
points $\wt{z}_1(n),\wt{z}_2(n),\wt{z}_3(n)$ in $\B(q, \frac{1}{n}) \cap
\widetilde{S}_n$.  Since the
intrinsic distances between any two of
these points is diverging to infinity and for each $i$,
$\wt{\Sigma}(n)(\wt{z}_i( n),\de_2) \subset \ov{B}_{\wt{\Sigma}(n)}(\wt{z}_i(n),\frac12)$,
then the disks $\{\wt{\Sigma}(n)(\wt{z}_i( n),\de_2)\mid i=1,2,3\}$
form a pairwise disjoint collection.

It follows that for $n$ large, the boundaries of the disks
$\{\wt{\Sigma}(n)(\wt{z}_i( n),\de_2)\mid i=1,2,3\}$
are contained in $\rth-\B(q,\delta_2/2)$.  Hence, by Corollary~\ref{cest},
there exists a small $\delta'>0$ such that
the components $\{\wh{\Sigma}(n,i)\mid i=1,2,3\}$ of $\wt{\Sigma}(n)(\wt{z}_i( n),\de_2)
\cap \oB(q,\de')$ containing the respective points  $\wt{z}_i( n)$ have second
fundamental forms bounded by a universal constant
and so, after possibly replacing $\de'$ by a
smaller positive number, these components
are disks which are graphical over their
projections to a plane passing through $q$.
After replacing by another subsequence and after reindexing,
a sequence of  pairs $\wh{\Sigma}(n,1)$,
$\wh{\Sigma}(n,2)$
of these graphs converges to a stable
compact $H$-disk $\cD$ passing through $q$, for
some value of $H$. Moreover, we will assume that
the inner products
of the  unit normal vectors of $\wh{\Sigma}(n,1)$ and
$\wh{\Sigma}(n,2)$ are positive, even when $H=0$.
Repeated applications of Corollary~\ref{cest} together with a
prolongation argument, as carried out in the proof of
Proposition~3.4 in~\cite{cm35} and in the proof of Proposition~8
in~\cite{mr13}, demonstrates that $\cD$ is contained in the
image of a complete
immersion   $f\colon F\looparrowright \rth$
of constant mean curvature $H\geq0$ of bounded norm
of the second fundamental form, for some
complete surface $F$. Indeed, after possibly
replacing by a subsequence and reindexing, the sequence of surfaces
$$\ov{B}_{\wt \S(n)}(\wt z_1(n),n)\cup \ov{B}_{\wt \S(n)}(\wt z_2(n),n)$$
converges smoothly to $f(F)$
with multiplicity at least two.

Since $f(F)$ is a limit of embedded surfaces, then $f$
satisfies the additional property that it is almost-embedded in the sense
that if $p_1\neq p_2$ are points in $F$  with the same image $p$ in $f(F)$,
then  images of small  intrinsic neighborhoods of $p_1,p_2$
locally lie on one side
of each other at $p$; note that by the maximum
principle for $H$-surfaces, this non-embeddedness property of $F$
can only occur if $H>0$ and the mean curvature vectors
at $p$ of these two respective neighborhoods are negatives of each other.
Since complete almost-embedded $H$-surfaces of
bounded norm of the second fundamental form  are properly
immersed in $\rth$ by Corollary~2.5 in~\cite{mt3}
(also see Theorem~6.1 in~\cite{mr7}), $f$ is a
proper immersion of $F$ into $\rth$.

If the universal cover of the limit surface
$f(F)$ is a stable $H$-surface, then $f(F)$ is a
plane~\cite{lor2,rst1} that intersects $\oB(5)$.
But if the image surface $f(F)$ is a plane $P$,
then for $n$ sufficiently large,
$\ov{B}_{\wt \S(n)}(\wt z_1(n),11)$ would become arbitrarily close to a planar
disk of radius 11 contained in $P$ and which intersects $\oB(5)$,
and so  $\partial B_{\wt \S(n)}(\wt z_1(n),11)\subset [\rth-\oB(5)]$.
Since the intrinsic
distance between $\wt z_1(n)$ and $\wt z_2(n)$ is going to
infinity as $n$ goes to infinity, this implies that
$\wt z_1(n)$ and $\wt z_2(n)$ cannot be connected by a curve
in $\wt \S(n) \cap \oB(5)$.
This contradiction would prove the proposition.
Thus, it suffices to prove the following claim.
\begin{claim} \label{claim2.9}
The universal cover of $F$ is stable. \end{claim}

\begin{proof}
Let $\Pi\colon \langle\wt{F},\wt{p}\rangle \to \langle{F},p\rangle$
denote the universal cover of the pointed surface
$ \langle{F},p\rangle$, where $f(\Pi (\wt{p}))=f(p)=q\in \cD$.
The proof of this claim uses standard arguments to construct
a non-zero Jacobi field on an arbitrary open connected subset
$\Omega\subset \wt F$ with
compact closure. Since the norm of the second fundamental
form of $\wt F$ is bounded, there is
a normal disk bundle $\mathcal N$ of fixed radius that
submerses in $\rth$. Let $i\colon \mathcal N\to
\rth$ denote the submersion, then we give $\mathcal N$ the
flat metric induced by $i$. We consider
the surface $\wt F$ to be the zero section of $\mathcal N$
and the map $\Gamma\colon \mathcal N\to \wt F$
given by the nearest point projection is smooth. Let $\xi$
denote the unit normal vector field to $\wt F$.

Let $\Delta_n$ be the lift (preimage) in $\mathcal N$ of $$[\ov{B}_{\wt \S(n)}(\wt z_1(n),n)
\cup \ov{B}_{\wt \S(n)}(\wt z_2(n),n)]\cap i(\mathcal N) $$
via the submersion $i$, where we can make choices of preimages
$z_1'(n)\in i^{-1}(\wt{z}_1(n)),$ $ z_2'(n)\in i^{-1}(\wt{z}_2(n))$
that converge to the same point $q' \in i^{-1}(q)$.
By the nature of the convergence, $\Delta_n$ converges
smoothly to $\wt F$ in $\mathcal N$ with
multiplicity at least two. Namely,  each point
$p\in \wt F$ has a neighborhood $U_p$ which is the
uniform limit of at least two disjoint domains
$U_{p,1}(n)$, $U_{p,2}(n)$ in $\Delta_n$. Each of
these domains is a graph over $U_p$ via the nearest
point projection $\Gamma$ and the normal
vectors of such graphs at related points have inner products
converging to 1 as $n$ goes to infinity.

Let $\Omega\subset \wt F$ be an arbitrary open connected
subset  with compact closure and let
$\wt\Omega\subset \wt F$ be a precompact, simply-connected
domain containing $\Omega$ and the point $q'$. The usual
holonomy construction and the convergence  with multiplicity
at least two gives that there exists
two disjoint domains in $\Omega_1(n)$, $\Omega_2(n)$
in $\Delta_n$ such that the following holds. Each
$\Omega_i(n)$ is a graph over $\wt \Omega$ via the
nearest point projection and $z'_1(n)\in \Omega_1(n)$,  $z'_2(n)\in \Omega_2(n)$. Namely, there exists
$u_i^n\colon \wt \Omega\to \R$ such that
\[
\Omega_i(n)=\{p+u^n_i(p)\xi(p) \text{ with } p\in \Omega\}.
\]
Because $\Omega_1(n)$ and $\Omega_2(n)$ are disjoint,
we can assume that $u_2^n>u_1^n$. Moreover, by our previous choice of the points
$\wt z_1(n), \wt z_2(n) $,
we may assume that
the unit normal  vectors of $\Omega_1(n)$ and of
$\Omega_2(n)$  at corresponding points
over points of $\Omega$ have positive inner
products converging to 1 as $n$ goes to infinity.
A standard compactness argument
using the Harnack inequality shows that the positive
function  $$\frac{u_2^n-u_1^n}{u_2^n(q')-u_1^n(q')}$$
converges to a positive Jacobi function $w$ over $\wt \Omega$.
Thus $w|_{\Omega}$ is a positive Jacobi
function over $\Omega$ which implies that $\Omega$ is stable.
Since $\Omega$ was an arbitrary precompact domain
in $F$, this finishes the proof of the claim.
\end{proof}
As mentioned previously, Claim~\ref{claim2.9} completes the proof of the proposition.
\end{proof}

\subsection{The function $a_\delta$.}
For the remainder of Section~2, $\Si$ will be the
$H$-disk in the statement of Theorem~\ref{thm1.1}.
We claim that if the   theorem holds whenever $\Si$ is a smooth
$H$-disk,
then it  holds in general.  To see this claim holds, assume
$\de_1\in(0,\frac12)$ is such that
 Theorem~\ref{thm1.1} holds for smooth $H$-disks
and let  $\Sigma$ be
an $H$-disk which is non-smooth along its boundary. Suppose
$\ov{B}_\Si(x,R) \subset \Si-\partial \Si$ and consider the conditions below:
\ben[C1.]  \item $\Si (x,\delta_1 R)$ is a disk with piecewise smooth boundary
$\partial \Sigma(\vec{0},\delta_1 R)\subset \partial \B(\de_1R)$. \item
$
 \Si (x, \delta_1 R) \subset \ov{B}_\Si (x, \frac{R}{2}).
$\een
Since the compact intrinsic ball $\ov{B}_\Si(x,R)$ is contained in
the interior of a smooth sub $H$-disk $\Si'\subset  (\Si-\partial \Si) $,
then $\ov{B}_\Si(x,R)=\ov{B}_{\Si'}(x,R)$ and
$\Si (x, \delta_1 R)=\Si' (x, \delta_1 R)$.
Using that   Theorem~\ref{thm1.1} holds for $\Si'$
then implies that the two conditions C1 and C2 hold.  So henceforth we will assume
that $\Si$ is a \underline{smooth}
$H$-disk.

 We claim that
for $\delta \in (0, \frac{1}{2})$, the function
$$G_{\de}(z)=\frac{d_\Si(z, \partial \Si)}{R(z,\delta)}\colon
\Si-\partial \Si \to (0,\infty)$$
is bounded on $\Si-\partial \Si$ and is equal to $1$ in some small neighborhood
of $\partial \Si$. To see this first note that if for
some $\ve >0, \; p \in \Si -\partial \Si$ has distance at least
$\ve$ from $\partial \Si$, then $R(p,\delta)$ is greater than some
positive constant that only depends on $\ve$ and $\Si$.  This is
because the norm of the second fundamental form of $\Si$ is bounded and so for any
$\ve^\prime<\ve$ sufficiently small, there exists $\ve''<\ve'$ such that
$\ov{B}_\Si(p, \ve^\prime)$ is a graph over its projection to its tangent
plane at $p$ with small norm of its
gradient and  $\Si(p,
\frac{\ve^\prime}{2}) \subset \ov{B}_\Si (p, \frac{\ve^\prime}{2}+\ve'')$,
with $\lim_{\ve'\to 0}\frac{\ve''}{\ve'}=0$.  Since $\delta <
\frac{1}{2}$, then $R(p, \delta)$ is bounded from below outside of
any small $\ve$-regular neighborhood of $\partial \Si$.  On the
other hand, since the geodesic curvature of $\partial \Si$ and the norm of the
second fundamental form are both bounded, then the same argument
shows that for some sufficiently small $\ve >0, R(p, \delta)$ is equal to
$d_\Si(p, \partial \Si)$, when $p \in \Si - \partial \Si$ and
$d_\Si(p, \partial \Si) < \ve$. This proves that the function $G$ is
bounded and is equal to $1$ in some
neighborhood of $\partial \Si$.

\begin{definition} Let
  $\delta \in (0,\frac{1}{2})$.  Then we define:
$$a_\delta = \sup_{z \in (\Si -\partial \Si)} \frac{d_\Sigma(z, \partial \Si)}{R(z, \delta)}
=\sup(G_{\de}).$$
\end{definition}

The next lemma and its proof correspond to Lemma~11 in~\cite{mr13}.

\begin{lemma} \label{lem10}
Let  $\delta^\prime \in(0,
\frac{1}{2})$.  If $a_{\delta^\prime} < c, \;c \in [2, \infty)$,
then Theorem~\ref{thm1.1} holds for $\Si$ with $\delta_1 =
\frac{\delta^\prime}{c}$.
\end{lemma}

\begin{proof} Suppose that for some $R>0$ and
$x\in \Si$, $\ov{B}_\Si(x,R) \subset \Si - \partial
\Si$. By definition of $a_{\delta'},$
$$ a_{\delta^\prime} \geq
\frac{d_\Sigma(x, \partial \Si)}{R(x, \delta^\prime)} >
\frac{R}{R(x, \delta^\prime)},$$ which implies that $R(x,
\delta^\prime)> \frac{R}{c}$.  Since  $R(x,
\delta^\prime)> \frac{R}{c}$, the definition of
$R(x,
\delta^\prime)$ implies that $\Si(x, \delta^\prime \frac{R}{c})$ is a disk and
$$
 \Si(x, \delta^\prime \frac{R}{c})\subset \ov{B}_\Si(x,
\frac{1}{2}\cdot  \frac{R}{c}) .$$  Thus,
since  $\displaystyle \Si(x, \frac{\delta^\prime }{c}R)
=\Si(x, \delta^\prime \frac{R}{c})$ and $\displaystyle \frac{R}{c}<R$,
we conclude that $\displaystyle \Si(x, \frac{\delta^\prime }{c}R)$ is a disk and
$$
 \Si(x, \frac{\delta^\prime }{c}R)  \subset \ov{B}_\Si(x,
\frac{1}{2}R),$$
and so, Theorem~\ref{thm1.1}
holds for $\Si$ with $\delta_1 = \frac{\delta^\prime}{c}$. \end{proof}

\subsection {Locating the smallest scale which is
 not $\delta$ weakly chord arc.}

The proof of the next lemma uses a standard technique for finding a
smallest scale  for which some property holds on a
surface. The property we are considering here is that of being
$\delta$ weakly chord arc. In this case we take the proof directly
from the proof of the similar Lemma~3.39 in~\cite{cm35}.

\begin{lemma} \label{cm3.39} Let  $\delta \in
  (0,\frac{1}{2})$.  Then there exists a point $y \in \Si$
and a number $R_1 > 0$ such that:
\begin{enumerate}[1.]
\item $a_\delta R_1 < \frac{1}{2} d_\Si(y, \partial \Si)$.
\item $R(x, \delta) > R_1$ for every $x \in \ov{B}_\Si(y, a_\delta R_1)$.
\item $\ov{B}_\Si(y, 5R_1)$ is \underline{not} $\delta$ weakly chord arc.
\end{enumerate}
\end{lemma}

\begin{proof} Recall the function $G_{\de}$ on $\Si \, -\, \partial \Si$ is defined by
$G_{\de}(x) = d_\Si (x, \partial \Si)/R(x, \delta)$ and extends to a
bounded function on $\Sigma$ which has a constant value $1$ near $\partial \Si$.
Thus, $a_\delta = \sup( G_{\de})$ is a finite number that is greater than or equal to 1. Choose $y$ to
be a point in $\Si-\partial\Si$ so that $G_{\de}(y)$ is greater than
$\displaystyle \frac{a_\delta}{2}$. Hence, if we define
$d_\partial = d_\Si(y, \partial \Si)$, then
$\displaystyle \frac{a_\delta}{2}< \frac{d_\partial}{R(y, \delta)}=G_{\de}(y)$,
or equivalently,
\begin{equation} \label{eq:2.13} a_{\de} R(y,\de)<2d_{\partial}.\end{equation}

Now choose $R_1 = R(y,\delta)/4$ and we will show this definition
of $R_1$ satisfies the statements in the lemma.  This value of $R_1$ and \eqref{eq:2.13}
give the inequality $a_\delta R_1 < \frac{1}{2} d_\partial,$ which
is statement 1 in the lemma.  By definition of $R_1$, $R(y,
\delta)=4R_1$ and by the definition of $R(y, \delta)$ as a supremum,
the ball $\ov{B}_\Si(y, 5 R_1)$ is not $\delta$ weakly chord arc, which
proves statement 3.

By statement 1, $a_\delta R_1 < \frac{1}{2} d_\partial$, and so,
$\ov{B}_\Si(y, a_\delta R_1) \subset \ov{B}_\Si(y, d_\partial /2)$.  So if we
check that statement 2 holds for points in $\ov{B}_{\Si}(y,
d_\partial/2)$, then statement 2 holds.  If $x \in \ov{B}_\Si(y,
d_\partial/2)$, then by the triangle inequality, $d_\partial/2 \leq
d_\Si(x, \partial \Si )$.  This inequality, the definition of $G_{\de}$
and the choice of $y$ give the inequalities
\[
\frac{d_\partial}{2R(x, \delta)}\leq \frac{d_\Si(x, \partial \Si )}{R(x, \delta)}
=G_{\de}(x) \leq a_\delta < 2G_{\de}(y) = \frac{2
d_\partial}{R (y, \delta)}.
\]
 Therefore, $R(x, \delta) >
R(y,\delta)/4=R_1$.  This completes the proof of statement 2 and the
lemma now follows.
\end{proof}

\subsection{The proof of
Theorem~\ref{thm1.1}.}
We now prove Theorem~\ref{thm1.1}.   By Lemma~\ref{lem10}, we just
need to prove that $a_\delta$ is bounded independently of $\Sigma$
for some fixed constant $\delta \in (0, \frac{1}{2})$.

Let $\delta=\delta_2$, where $\delta_2$ is given Proposition~\ref{cm2.1}.
We now prove $a_\delta$ is bounded from above by
 $C_b$, where $C_b$ is given in Proposition~\ref{propcm3.4}.
 Suppose there exists a $\Sigma$ with $a_\delta >
C_b$.

By the Lemma~\ref{cm3.39}, there exist a point $y \in \Si$ and an
$R_1$, such that:
\begin{enumerate}[1.]
\item $\ov{B}_\Si(y, a_\delta R_1) \subset \ov{B}_\Si(y, \frac{1}{2}d_\Sigma
(y, \partial \Si)).$
\item $R(x, \delta) > R_1$ for every $x \in \ov{B}_\Si(y, a_\delta R_1).$
\item $\ov{B}_\Si(y, 5R_1)$ is \underline{not} $\delta$ weakly chord arc.
\end{enumerate}

By definition of $R(x, \delta)$ and statement 2, we have that
$\ov{B}_\Si(x, R_1)$ is $\delta$ weakly chord arc for every $x \in
\ov{B}_\Si(y, C_b R_1) \subset \ov{B}_\Si(y, a_\delta R_1).$  But
Proposition~\ref{propcm3.4} implies that $\ov{B}_\Si(y, 5 R_1)$ is
$\delta$ weakly chord arc, contradicting statement 3 above. This
contradiction completes the proof  of Theorem~\ref{thm1.1}.

\section{Applications of  Theorem~\ref{thm1.1}.} \label{sec4}

The next result gives a useful intrinsic one-sided curvature
estimate; its proof uses Theorem~\ref{thm1.1} and
the extrinsic  one-sided curvature
estimate given in Theorem~\ref{th}.  This next result is also stated as Theorem~4.5
in~\cite{mt9}; in the case that $H=0$, the
next theorem follows from Corollary~0.8 in~\cite{cm35}.

\begin{theorem}[Intrinsic one-sided curvature estimate for $H$-disks] \label{TH}
There exist $\ve_I\in(0,\frac{1}{2})$
and $C_I \geq 2 \sqrt{2}$ such that for any $R>0$, the following holds.
Let $\cD$ be an $H$-disk such that $$\cD\cap \B(R)\cap\{x_3=0\}=\O $$
and $x\in \cD \cap \B(\ve_I R)$, where $d_\cD(x,\partial \cD) \geq R$.  Then:
\begin{equation} \label{EQ1}
|A_{\cD}|(x)\leq \frac{C_I}{R}.
\end{equation} In particular,  $H< \frac{C_I}{R}$.
\end{theorem}
\begin{proof}
Let $\ve$, $C$ be the constants given in Theorem~\ref{th}
and let $\de_1$ be the constant
given in Theorem~\ref{thm1.1}. We next check that the
constants $\displaystyle \ve_I={\de_1 \ve}$
and $\displaystyle C_I=\frac{2C}{\de_1}$ satisfy the
conditions in the theorem.

Without loss of generality, we may assume
that  $x\in  \cD \cap \B(\ve_I R)\cap \{x_3>0\}$, where $d_\cD(x,\partial \cD) \geq R$.
By Theorem~\ref{thm1.1}, the surface $\cD':=\Sigma(x,\de_1 R)\subset \cD$ is
an $H$-disk with its boundary in $\partial \B(x,\de_1 R)$.
Since $\ve\in (0,1/2)$, then $d_{\rth} (x,\vec{0})< \ve \de_1R<\frac{\de_1}{2} R$, and so
the triangle inequality implies
$\B(\frac{\de_1}{2} R)\cap \partial \B(x,\de_1 R)=\O$.
Hence, $\Sigma(x,\de_1 R)$ must have its boundary
in $\rth-\B(\frac{\de_1}{2} R)$.
Therefore, the scaled surface $\frac{2}{\de_1}\cD'$
satisfies the conditions of the disk  described
in Theorem~\ref{th}; in other words,
$$(\frac{2}{\de_1}\cD')\cap \B(R)\cap\{x_3=0\}
=\O \quad \mbox{and} \quad \partial (\frac{2}{\de_1}\cD')\cap \B(R)\cap\{x_3>0\}=\O.$$
Since the scaled point
$\frac{2}{\de_1}x \in (\frac{2}{\de_1}\cD')\cap \B(\ve)\cap \{x_3>0\}$,
Theorem~\ref{th} gives
$|A_{\frac{2}{\de_1}\cD'}|(\frac{2}{\de_1}x)\leq \frac{C}{R}$, and so
$$|A_{\cD'}|(x)\leq \frac{2}{\de_1}\frac{C}{R}=\frac{C_I}{R},$$
which completes the proof of the theorem.
\end{proof}

The following result is a direct consequence of some of the
arguments in the proof of Theorem~\ref{thm1.1}.
\begin{theorem} \label{cor:appl}
Given $\ve>0$ and $m\in (0,\infty)$, there exists $R(m,\ve)>m$
such that the following holds. Let $\Sigma$ be a complete
$H$-surface with boundary such that for any $x\in \Sigma$,
\[
\inj_\Sigma(x)\geq \min\{\ve, d_\Sigma(x,\partial \Sigma)\}.
\]
If $\ov{B}_\Sigma(y,R) \subset \Sigma -\partial \Sigma$ with $R\geq R (m,\ve)$, then
\[
\Sigma(y, m) \subset \ov{B}_\Sigma(y,\frac{R}{2}).
\]
\end{theorem}
\begin{proof}
Arguing by contradiction, suppose there exist $\ve>0$ and $m\in (0,\infty)$
 such that for any $n> 1$, there exists a compact $H_n$-surface $\S(n)$ with
\[
\inj_{\S(n)}(x)\geq \min\{\ve, d_{\S(n)}(x,\partial \S(n))\},
\]
and $y_n\in \S(n)$ such that  $\ov{B}_\Sigma(y_n,n)\subset \Sigma -\partial \Sigma$ but
\[
\S(n)(y_n, m) \not \subset \ov{B}_{\S(n)}(y_n,\frac{n}{2}).
\]

We now follow the arguments in the proof of
Proposition~\ref{propcm3.4} to give a sketch of the proof,
leaving the details to the reader. Without loss of generality, after rescaling
by $\frac 1m$ and normalizing by translations, we can
assume that $m=1$ and that $y_n=\vec 0$.
Also we may assume that $\ve\in(0,1)$, since if the
theorem holds for a smaller positive choice of $\ve$,
then it holds for the original choice.

Since
$\S(n) (\vec 0,1)$ is path connected, we can find an embedded  path $\gamma _n
\subset\S(n) (\vec 0,1)$ starting at $\vec 0$ and ending at some
point of $\partial B_{\S(n)}(\vec 0,\frac{n}{2})$.
As $n$ goes to infinity, there exists
a subset $S_n=\{ z_n(1),\ldots
,z_n(k(n))\} \subset \gamma _n$ with
\begin{enumerate}[1.]
\item $\lim_{n\to\infty} k(n)=\infty$.
\item The intrinsic distance in $\S(n)$ between any two of the
points of $S_n$ tends to infinity as $n$ goes to infinity.
\end{enumerate}

The hypothesis $\inj_\Sigma(x)\geq \min\{\ve, d_\Sigma(x,\partial \Sigma)\}$
implies that if $x\in \S(n)$ with
$d_{\S(n)}(x,\partial \S(n))>\ve$, then
the geodesic ball $\ov{B}_{\S(n)} (x, \ve)$ is a disk.
Thus, Theorem~\ref{thm1.1} gives that
${\S(n)} (x,\delta_1 \ve)$ is a compact disk
with piecewise smooth boundary in $\partial \B(x, \delta_1\ve)$ and
$\S(n) (x,\delta_1 \ve)\subset \ov{B}_{\S(n)}(x,\ve/2)$.
In particular this is true for any $z\in S_n$. Recall that the number of disks
${\S(n)} (z,\delta_1\ve)$ centered at points
$z $ in $S_n$ diverges as $n \rightarrow
\infty$. Hence, after  reindexing and
replacing by a subsequence, there is a point
$q\in \ov\B(1)$, where the number
of points in $\B(q, \frac{1}{n}) \cap
S_n$ goes to infinity as $n$ goes to infinity.
In particular, we may assume that for each $n$, there exist
three distinct
points ${z}_1(n),{z}_2(n),{z}_3(n)$ in $\B(q, \frac{1}{n}) \cap
{S}_n$.  Since the
intrinsic distances between any two of
these points is diverging to infinity and
${{\S(n)}}({z}_i( n),\delta_1\ve) \subset \ov{B}_{\S(n)} ({z}_i(n), \ve/2)$,  for  $i\in \{1,2,3\}$,
then the disks $\{{\Sigma}(n)({z}_i( n),\delta_1\ve)\mid i=1,2,3\}$
form a pairwise disjoint collection.
Arguing similarly as in the prolongation argument in the proof of Proposition~\ref{propcm3.4},
in the limit we obtain a complete  stable minimal surface $F$ (which is a plane),
which can be used to prove that the points
${z}_1(n),{z}_2(n),{z}_3(n)$ can not be contained
in the embedded arc $\gamma_n\subset \oB(1)$. This gives a
contradiction and completes sketch of the proof of the theorem.
\end{proof}

\begin{remark}  {\em
Theorem~\ref{cor:appl} can be improved in various ways.
For example, Theorem~\ref{thm1.1} implies that
the hypothesis \[
\inj_\Sigma(x)\geq \min\{\ve, d_\Sigma(x,\partial \Sigma)\}.
\]
can be replace by the weaker condition that there exists an
$\ve>0$ such that for all $x\in \Sigma$
such that $d_\Sigma(x,\partial \Sigma)>\ve$, the intrinsic
ball $\ov{B}_{\S}(x,\ve)$ is contained in a simply-connected subdomain
of $\Sigma$.
}
\end{remark}

\vspace{.3cm}
\center{William H. Meeks, III at profmeeks@gmail.com\\
Mathematics Department, University of Massachusetts, Amherst, MA 01003}
\center{Giuseppe Tinaglia at giuseppe.tinaglia@kcl.ac.uk\\ Department of
Mathematics, King's College London,
London, WC2R 2LS, U.K.}

\bibliographystyle{plain}
\bibliography{bill}

\ed